\documentclass{amsart}
\usepackage[utf8]{inputenc}
\usepackage{games}
\usepackage{amssymb}
\usepackage{url}

\newtheorem{theorem}{Theorem}
\newtheorem{corollary}{Corollary}

\newtheorem{lemma}{Lemma}
\theoremstyle{definition}
\newtheorem{definition}{Definition}
\newtheorem{example}{Example}
\newtheorem{remark}{Remark}

\newtheorem{notation}{Notation}

\newcommand{\setof}[1]{\left\{#1\right\}}
\newcommand{\suchthat}{~\colon}
\newcommand{\abs}[1]{\left\vert #1 \right\vert}

\newcommand{\I}{{\boldsymbol{I}}}
\newcommand{\II}{{\boldsymbol{I}\kern-0.05cm\boldsymbol{I}}}
\newcommand{\res}{\mathord \upharpoonright}
\newcommand{\conc}{{}^\smallfrown}

\newcommand{\x}{\mathbf{x}}
\newcommand{\y}{\mathbf{y}}
\newcommand{\R}{\mathbb{R}}
\newcommand{\N}{\mathbb{N}}
\newcommand{\ww}{\omega^\omega}

\newcommand{\ad}{\mathsf{AD}}
\newcommand{\pd}{\mathsf{PD}}
\newcommand{\concat}{{}^\smallfrown}

\newcommand{\diam}{\text{diam}}

\newcommand{\ceil}[1]{\left\lceil#1\right\rceil}

\newcommand{\dom}{\text{dom}}
\DeclareMathOperator{\HD}{HD}
\newcommand{\ac}{\mathsf{AC}}

\newcommand{\zf}{\mathsf{ZF}}

\newcommand{\bP}{\boldsymbol{\Pi}}

\newcommand{\dc}{\mathsf{DC}}
\newcommand{\bS}{\boldsymbol{\Sigma}}
\newcommand{\bD}{\boldsymbol{\Delta}}
\newcommand{\lh}[1]{\abs{#1}}

\newcommand{\cylinder}[1]{[#1]}
\newcommand{\cof}{\text{cof}}

\title{Hausdorff Dimension Regularity Properties and Games}
\author{Logan Crone}
\address{Logan Crone, University of North Texas, Department of Mathematics, 1155 Union Circle
  \#311430, Denton, TX 76203-5017, USA}
\email{logancrone@my.unt.edu}
\author{Lior Fishman}
\address{Lior Fishman, University of North Texas, Department of Mathematics, 1155 Union Circle \#311430, Denton, TX 76203-5017, USA}
\email{lior.fishman@unt.edu}
\author{Stephen Jackson}
\address{Stephen Jackson, University of North Texas, Department of Mathematics, 1155 Union Circle \#311430, Denton, TX 76203-5017, USA}
\email{stephen.jackson@unt.edu}

\begin{document}

\begin{abstract}
The Hausdorff $\delta$-dimension game was introduced in \cite{DFSU_Arxiv}
and shown to characterize sets in $\R^d$ having Hausdorff
dimension $\leq \delta$. We introduce a variation of this game which also characterizes Hausdorff dimension
and for which we are able to prove an unfolding
result similar to the basic unfolding property for the Banach-Mazur game for category.
We use this to derive a number of consequences for Hausdorff dimension. We show that under $\ad$ any wellordered
union of sets each of which has Hausdorff dimension $\leq \delta$ has dimension $\leq \delta$. We establish a
continuous uniformization result for Hausdorff dimension. The unfolded game also provides a new proof
that every $\bS^1_1$ set of Hausdorff dimension $\geq \delta$ contains a compact subset of dimension
$\geq \delta'$ for any $\delta'<\delta$, and this result generalizes to arbitrary sets under $\ad$.
\end{abstract}

\maketitle

\section{Introduction} \label{sec:intro}

Category, measure, and Hausdorff dimension are three fundamental notions of largeness
for sets in a Polish space (for Hausdorff dimension one most commonly restricts to
subsets of $\R^d$). In the case of category the notion is connected to a well-known game,
the Banach-Mazur or $**$-game (see for example \cite{KechrisBook} for a discussion of the
game and related notions; we note that for the Banach-Mazur game $G^{**}(A)$ for a set $A$
it is conventional to have player $\II$
being the player trying  to get into the set $A$). For example, in any Polish space $X$, a set $A\subseteq X$ is comeager
iff $\II$ has a winning strategy in the Banach-Mazur game $G^{**}(A)$, and $\I$ has a winning
strategy iff there is a neighborhood on which $A$ is meager. An important aspect of this game is that it permits an
{\em unfolding}. By this we mean that if $A=\dom(R)$ where $R\subseteq X\times Y$, then if $\I$
has a winning strategy in the game $G^{**}(R)$ then $\I$ has a winning strategy in the game
$G^{**}(A)$. Assuming the game is determined, this says that if $\II$ can win the game $G^{**}(A)$,
then $\II$ can actually win the game $G^{**}(R)$ in which $\II$ is not only produces an $x \in A$
but also a pair $(x,y)\in R$, that is, where $y$ ``witnesses'' that $x \in A$.

This unfolding phenomenon
for the $**$-game has many applications to category. For example, since $\bS^1_1$ sets $A\subseteq X$
are projections of closed sets $F\subseteq X\times \ww$, this reduces the Banach-Mazur game for $\bS^1_1$ sets
to the game for closed sets, which are determined in $\zf$. This gives a proof of the fact that
every $\bS^1_1$ set in a Polish space has the Baire property. Another application of the unfolding is
to show continuous uniformizations on comeager sets. Namely, suppose $R\subseteq X\times Y$ and
$A =\dom(R)$ is comeager. If we assume $\ad$, then there is a comeager set $C\subseteq A$ and a continuous function
$f \colon C\to Y$ which {\em uniformizes} $R$, that is, for all $x \in A$ we have $R(x,f(x))$.
Working just in $\zf$ we get that if $R$ is $\bS^1_1$ then there is a continuous uniformization on
a comeager set (this requires unfolding the game on $R$ to a closed set $F\subseteq X\times Y\times \ww$).
Yet another application of unfolding is to establish the full additivity of category under $\ad$.
By this we mean the statement that a wellordered union on meager sets is meager. The most common proof
given for this uses the analog of Fubini's theorem for category, the Kuratowski-Ulam theorem.
However, a different proof can be given using the unfolded game. This is important as there is no Fubini
theorem for Hausdorff $\delta$-dimension measure, and we wish to establish this additivity result for Hausdorff dimension
(Theorem~\ref{thm:wou}).

In \cite{Martin1998} (see also \cite{RosendalNotes2009}) the {\em Measure game} was introduced which was shown to
characterize Lebesgue measure
in a manner similar to how the Banach-Mazur game characterizes category.
In \cite{CFJSS2019} a variation of this game was introduced and an unfolding result for it was proved.
This game analysis had several applications. Aside from giving new proofs of some classical results such as the
Borel-Cantelli lemma, a strong form of the R\'{e}nyi-Lamperti lemma of probability theory was shown using the game.

In \cite{DFSU_Arxiv} a game, the Hausdorff $\delta$-dimension game was introduced, and it was shown that this game
characterizes when a set $A\subseteq \R^d$ has Hausdorff dimension $\HD(A)\leq \delta$.
More precisely, if $\I$ wins the game then $\HD(A)\geq \delta$ and if $\II$ wins the game then
$\HD(A)\leq \delta$. In this paper we introduce a variation of this game which we show also
characterizes Hausdorff dimension in this manner, and for which we are able to prove an
unfolding result (Theorem~\ref{thm:unfthm}). As with measure and category, this has a number of consequences.
This gives a new proof of the basic regularity result that every $\bS^1_1$ set $A\subseteq \R^d$ with
$\HD(A)\geq \delta$, if $\delta' <\delta$ then $A$  contains a contains a compact set $K$
with $\HD(K)\geq \delta'$.
The classical proof of this fact uses an ``increasing sets lemma'' for Hausdorff $\delta$-measure
(see Theorems~47 and 48 of \cite{RogersBook}).
Moreover, this result extends to other pointclasses assuming the determinacy
of the corresponding games. For example, assuming $\bP^1_1$-determinacy we get the same regularity
result for $\bS^1_2$ sets. We are able to prove continuous uniformization theorems, (see
Theorem~\ref{thm:cut} and the following remarks)
again assuming the determinacy of the relevant games. Finally, using the unfolded game
we are able to show that under $\ad$ we have full additivity for Hausdorff dimension $\leq \delta$
sets. That is, any wellordered union (of any length) of sets each of which has Hausdorff dimension
$\leq \delta$ has Hausdorff dimension $\leq \delta$. This complements the corresponding
results for category and measure, which are known theorems from $\ad$.

Throughout, we will be working in a Euclidean space $\R^d$. We let $\mathcal{H}^s$ denote
$s$-dimensional Hausdorff measure on $\R^d$. For $A\subseteq \R^d$ we let
$\HD(A)$ denote the Hausdorff dimension of $A$. This is defined for all sets $A\subseteq \R^d$,
and $0 \leq \HD(A)\leq d$. 
We recall that $\mathcal{H}^s$ is a Borel measure on $\R^d$,
but it is not $\sigma$-finite (unless $s=d$). We let $\omega=\N$ denote the natural numbers
and $\ww$ denote the Baire space (set of sequences of natural numbers) with the usual
product of the discrete topologies on $\omega$.

The following theorem is a well-known tool in the theory of Hausdorff dimension,
and is also central to our arguments. We include a proof partly for the sake of
completeness, and also because we wish to be able to use our results  in models
of determinacy where $\ac$ fails. In the following proof we show that only
countable choice $\ac_\omega$ is needed, and thus we can in particular use this result in any model of $\zf+\dc$.

\begin{theorem}[Rogers-Taylor-Tricot \cite{RogersBook}). ($\zf+\mathsf{AC_\omega}$]\label{thm:rtt}
Let $\mu$ be a Borel probability measure on $\mathbb{R}^d$.
\begin{enumerate}
\item[\textbf{(i)}] If $A \subseteq \mathbb{R}^d$ and $\displaystyle \limsup_{r \to 0} \frac{\mu(B(x, r))}{r^s} < m$
for every $x \in A$, then 
\[\mathcal{H}^s(A) \geq m^{-1}\mu^*(A).\]
\item[\textbf{(ii)}] If $A \subseteq \mathbb{R}^d$ and $\displaystyle \limsup_{r \to 0} \frac{\mu(B(x, r))}{r^s} >m $
for every $x \in A$, then 
\[\mathcal{H}^s(A) \leq c_d m^{-1}\mu^*(A)\] where $c_d$ is a constant depending only on $d$.
\end{enumerate}
In these statements, $\mathcal{H}^s(A)$ refers to the Hausdorff $s$-dimensional outer measure
of $A$, and $\mu^*(A)$ refers to the outer $\mu$-measure of $A$. 
\end{theorem}
\begin{proof}
\textbf{(i)} Let $A_\epsilon = \setof{x \in A \suchthat \sup_{0 < r <
\epsilon} \frac{\mu(B(x, r))}{r^s} < m}$, and note that
$\bigcup_{\epsilon>0} A_\epsilon = A$ and that $\mu^*(A_\epsilon) \to
\mu^*(A)$ as $\epsilon \to 0$ (note that the $A_\epsilon$ are increasing as $\epsilon \to 0$ and
for any Borel probabillity measure $\mu$ and increasing sequence of sets $C_n$
we have that $\mu^*(\bigcup_nC_n)=\lim_n \mu^*(C_n)$).
Fix $\epsilon > 0$ and let
$\setof{B_i}_{i \in \omega}$ be a cover of $A_\epsilon$.  Suppose
further that each $B_i$ intersects $A_\epsilon$ and that each $B_i$
has diameter $r_i < \epsilon$.  For each $i$, let $x_i \in B_i \cap
A_\epsilon$, then we have that the sequence $\setof{B(x_i, r_i)}_{i
\in \omega}$ is also a cover of $A_\epsilon$ and that for each $i$,
$\mu(B(x_i, r_i)) < m r_i^s$.  Thus
\[\sum_i r_i^s \geq m^{-1} \sum_i \mu(B(x_i, r_i)) \geq m^{-1} \mu^*(A_\epsilon).\]

Since $\mathcal{H}^s_\epsilon(A_\epsilon)$ is the greatest lower bound of all the quantities
$\sum_i r_i^s$ for such covers of $A_\epsilon$, and we know $m^{-1} \mu^*(A_\epsilon)$ is a
fixed lower bound for each such sum, we have that
\[\mathcal{H}^s(A) \geq \mathcal{H}^s_\epsilon(A) \geq \mathcal{H}^s_\epsilon(A_\epsilon) \geq m^{-1} \mu^*(A_\epsilon)\]

Since $\mu^*(A_\epsilon) \to \mu^*(A)$, we have the desired inequality.

\textbf{(ii)} Let $\epsilon>0$ and let $U$ be any open set containing $A$.
Let \[\mathcal{Q}_m = \setof{Q \subseteq \mathbb{R}^d \suchthat Q~\text{is a dyadic cube with side length}~\frac{1}{2^m}}.\]
For each $m$ and each $x \in A$, let 
\begin{equation*}
\begin{split}\mathcal{S}_m(x)= \big\{Q \in \mathcal{Q}_m \suchthat \exists r>0,
&~ r/2 < \diam(Q)\leq r < \epsilon ~\wedge\\
& Q \cap B(x, r) \neq \emptyset ~\wedge\\
& B(x, 2r) \subseteq U ~\wedge\\
& \mu(B(x, r)) > mr^s
\big\}
\end{split}
\end{equation*}
and let
\begin{equation*}
\begin{split}\mathcal{T}_m(x)= \big\{Q \in \mathcal{S}_m(x) \suchthat \forall Q' \in \mathcal{S}_m(x),~ \mu(Q') \leq \mu(Q)
\big\}
\end{split}
\end{equation*}

We note that $\mathcal{S}_m(x)$ is always finite, and since $x \in A$
there must be some $m$ so that $\mathcal{S}_m(x)$ is nonempty, and so
for some $m$, $\mathcal{T}_m(x)$ is also finite and nonempty.  For
each $x$, let $m(x)$ be minimal so that $T_{m(x)}(x)$ is nonempty, and
let $Q_x \in \mathcal{T}_{m(x)}(x)$ be of minimal index in some fixed
enumeration of the dyadic cubes.  For each dyadic cube $Q$ (of which
there are only countably many) which is equal to some $Q_x$, we choose
$x(Q)$ so that $Q = Q_{x(Q)}$ and choose some witness $r(Q)>0$ to the
fact that $Q \in S_{m(x(Q))}(x(Q))$.

For such $Q$, we have $r(Q)/2<\diam(Q)$, thus we know that the side
length of $Q$ is $\frac{1}{2^{m(x(Q))}} \geq \frac{r(Q)}{2\sqrt{d}}$.
Let $N_d=\ceil{1+12\sqrt{d}}$ and note that if $k$ is the side length
of $Q$, then $kN_d\geq r(Q)\frac{1+12\sqrt{d}}{2\sqrt{d}} > 6r(Q)$ and
thus $B(x(Q), r(Q)) \subseteq Q^*$, where $Q^*$ is $Q$ scaled by
$N_d$. and thus $B(x(Q), r(Q))$ can be covered by $N_d^d$ dyadic cubes
of the same side length as $Q$, of which $Q$ has maximal $\mu$-measure
(since we can discard any dyadic cubes from the cover which do not
intersect $B(x(Q), r(Q))$.

So for each dyadic cube $Q=Q_{x(Q)}$, we have
\[\mu(Q) \geq N_d^{-d}\mu(B(x(Q), r(Q))) > N_d^{-d} m r(Q)^s\]
and since each such $Q$ is contained in $U$, we have
\[\mu(U) \geq \sum_{Q} \mu(Q) \geq N_d^{-d} m \sum_{Q} r(Q)^s.\]
Now also the collection of cubes $Q^*$ form a cover of $A$ (since if $Q=Q_x$,
the enlarged $Q^*$ must contain $x$, even if $x \neq x(Q)$).  Since each $Q^*$ is
covered by $N_d^{d}$ translates of $Q$, and since $\diam(Q) \leq r(Q) < \epsilon$, we have
\[\mathcal{H}^s_\epsilon(A) \leq N_d^d \sum_{Q} r(Q)^s \leq N_d^{2d} m^{-1} \mu(U).\]

Thus by taking a $\sup$ as $\epsilon \to 0$, we have 
\[H^s(A) \leq N_d^{2d} m^{-1} \mu(U)\] for any open set $U$ containing $A$.  Thus finally we have
\[\mathcal{H}^s(A) \leq  \ceil{1+12\sqrt{d}} m^{-1} \mu^*(A)\]
\end{proof}

\section{The revised Hausdorff dimension game}

As we mentioned before, the Hausdorff $\delta$-dimension game was introduced in \cite{DFSU_Arxiv}.
Here we define a variation of the game, the main difference is that we use not a single $\beta$,
but a sequence $\beta_i$ which goes to $0$ sufficiently slowly. Using a sequence of the $\beta_i$
does not  affect the fact that the game characterizes Hausdorff dimension (as Theorems~\ref{thm:p1}
and \ref{thm:p2} show), but seems important in our argument for the unfolding (Theorem~\ref{thm:unfthm}). 

\begin{definition}
Let $d \geq 1$ be an integer and fix $\rho_0>0$, $0 < \beta_{i+1}\leq\beta_i < \frac{1}{2}$
be so that $\lim_{i \to \infty} \beta_i = 0$ satisfying 
\begin{equation}\label{eqn:betaspeed}
\forall \eta>0\ \exists n_0\ \forall n \geq n_0\ \beta_n \geq \prod_{i<n} \beta_i^\eta.
\end{equation}
Define $\rho_n=\left(\prod_{i<n} \beta_i\right) \rho_0$. Let $A \subseteq \mathbb{R}^d$.
Let $0 <\delta \leq d$.  The $\delta$-Hausdorff dimension game with target set $A$ is following game:
\begin{center}
\begin{games}

\game[game label={$G_{\vec\beta}^\delta (A)$}, dots]{
{$F_0$}{$x_0$}{$F_1$}{$x_1$}{$F_2$}{$x_2$}{$F_3$}{$x_3$}}
\end{games}
\end{center}
where player $\I$ must follow the rules
\begin{itemize}
\item $F_i$ is a finite set of points in $\mathbb{Q}^d$.
\item $F_i$ is $3\rho_i$ separated. 
\item $F_{i+1} \subseteq B(x_i, (1-\beta_{i})\rho_i)$.
\item There exists some $c>0$ so that $\displaystyle \limsup_{n \to \infty}
\frac{\prod_{i < n} \lh{F_i}^{-1}}{{\prod_{i<n}\beta_i^\delta}} \leq c$.
\end{itemize}
and player $\II$ must simply play so that $x_i \in F_i$.  Provided the players meet these requirements,
player $\I$ wins if and only if $\lim_{n \to \infty} x_n \in A$.
Note that the last ``rule'' for player $\I$ is a limiting condition on the
average number of choices offered to player $\II$.
\end{definition}

\begin{remark}
In the original game of \cite{DFSU_Arxiv}, the points which player $\I$ plays need not be
rational.  It turns out that to prove the two theorems characterizing Hausdorff dimension
(Theorems~\ref{thm:p1} and \ref{thm:p2}) it suffices to use the rational version above, which of course
is determined from $\ad$. It is not clear that the the rational and real versions of the game are
equivalent, however. For sets $A$
for which the games are determined, the winning players must agree for
$\delta \neq \HD(A)$, but even though the games are determined, it seems possible that they
may disagree on who wins at $\delta=\HD(A)$. 
\end{remark}

\begin{theorem} \label{thm:p1}
If player $\I$ has a winning strategy in the $\delta$-Hausdorff dimension game,
then there is a compact $K\subseteq A$ with $\HD(K)\geq \delta$.
\end{theorem}

\begin{proof}
Suppose $\sigma$ is a winning strategy for player $\I$ in the $\delta$-Hausdorff dimension game.
Define a finitely splitting tree $T$ by
\[T=\setof{(x_0, \dots, x_n) \suchthat \forall i\ x_i \in \sigma(x_0, \dots, x_{i-1})}\]
Define a map $\pi\colon [T] \to \mathbb{R}^d$ by $\pi(x_0, \dots, x_n, \dots) = \lim_{n \to \infty} x_n$,
which is clearly continuous.  Define a probability measure $\mu$ on $[T]$ by 
\[\mu(\cylinder{(x_0, \dots, x_n)}) = \prod_{i< n} \abs{\sigma(x_0, \dots, x_i)}^{-1}\]
and let $\mu$ also denote the push-forward measure of $\mu$ on $\mathbb{R}^d$.
Since $\sigma$ is winning, $\pi([T]) \subseteq A$, and since $T$ is finitely splitting, $[T]$ is compact,
and since $\pi$ is continuous, $K=\pi([T])$ is compact. We show that $\HD(K)\geq \delta$. It suffices
to fix $\gamma<\delta$ and show that $\HD(K)\geq \gamma$. 
Now for $x \in K$, we have $x=\pi(x_0, \dots, x_n, \dots)$.  We compute
\begin{equation*}
\begin{split}
\limsup_{r \to 0}\frac{\mu(B(x, r))}{r^\gamma} &\leq \limsup_{n \to \infty}\frac{\mu(B(x, \rho_n))}{{\rho_{n+1}}^\gamma}\\
&\leq \limsup_{n \to \infty}\frac{\mu(B(x_n, \rho_n))}{{\rho_{n+1}}^\gamma}\ \text{by the separation rule}\\
&= \limsup_{n \to \infty}\frac{\mu(\cylinder{(x_0, \dots, x_n)})}
{\left(\prod_{i \leq n} {\beta_i}^\gamma\right){\rho_0}^\gamma}\\
&= \limsup_{n \to \infty}\left(\frac{\prod_{i< n} \abs{\sigma(x_0, \dots, x_i)}^{-1}}{\prod_{i \leq n}
{\beta_i}^\gamma}\right) \frac{1}{{\rho_0}^\gamma}\\
&
=\limsup_{n \to \infty}\left(\frac{\prod_{i< n} \abs{\sigma(x_0, \dots, x_i)}^{-1}}{\prod_{i < n}
{\beta_i}^\gamma} \right) \frac{1}{\beta_n^\gamma  {\rho_0}^\gamma}
\\ &
\leq \limsup_{n \to \infty}\left(\frac{\prod_{i< n} \abs{\sigma(x_0, \dots, x_i)}^{-1}}{\prod_{i < n}
{\beta_i}^\gamma} \right) \frac{1}{{\prod_{i<n}\beta_i}^{\eta\gamma}  {\rho_0}^\gamma}
\\ & 
\leq \limsup_{n \to \infty}\left(\frac{\prod_{i< n} \abs{\sigma(x_0, \dots, x_i)}^{-1}}{\prod_{i < n}
{\beta_i}^{\gamma(1+\eta)}} \right) \frac{1}{  {\rho_0}^\gamma}
\\ &
=0 \text{ by the limit rule on the number of moves, since $\gamma(1+\eta)>\delta$}
\end{split}
\end{equation*}
And so by Theorem~\ref{thm:rtt}, $H^\gamma(K)=\infty$ 
and thus $\HD(K) \geq \delta$.
\end{proof}

\begin{theorem} \label{thm:p2}
If player $\II$ has a winning strategy in the $\delta$-Hausdorff dimension game, then $\HD(A) \leq \delta$
\end{theorem}

\begin{proof}
Note first that for a ball $B \subseteq \mathbb{R}^d$ of radius $\rho$, any $3\beta\rho$
separated subset $E \subseteq B$ has size at most
\[\abs{E} \leq \ceil{\frac{4\sqrt{d}}{3\beta}}^d.\]
This can be seen by comparing the volumes of a cube of side length $4\rho$ (which contains $B$) and
the sums of the volumes of cubes centered on points in $E$ of side lengths $3\beta\rho/\sqrt{d}$, which
must be disjoint by hypothesis on $E$.

The actual bound is unimportant, we need that it depends only on $d$ and $\beta$.

Suppose now that $\tau$ is a winning strategy for player $\II$ in  the $\delta$-Hausdorff dimension game.
Let $\rho_n$, $\beta_n$ etc. be the parameters of the game.

Let $E_n$ be a maximal $\frac{1}{2}\rho_n$-separated subset of $\mathbb{Q}^d$. and let $\setof{E_n^{i}}_{0 \leq i < \ell}$
partition $E_n$ into $3\rho_n$-separated subsets.  Note that this can be done with a fixed $\ell$
which doesn't depend on $n$.  In fact, if each $E_n^i$ is a maximal $3\rho_n$-separated subset
of $E_n \setminus \bigcup_{j < i} E_n^j$, then we can bound $\ell$  by ${\ceil{14\sqrt{d}}}^d$, for example.

We will consider playing various legal subsets of $E^i_n$ at round $n$ of the game
against $\tau$ (subject to some restrictions).  We will then define a probability measure on the
tree of positions obtainable by playing this way, and then push this measure forward to $\mathbb{R}^d$
and apply Theorem~\ref{thm:rtt}.  More precisely, we will define for appropriate
$u \in (\ell \times \omega)^{<\omega}$ a position $p_u$ in the game, a ball $B_u$, and an
associated $\mu$-measure value for $N_u=\setof{z \in (\ell \times k)^\omega \suchthat u \subseteq z}$.
We will then push $\mu$ forward via the function which computes the resulting point in the game.
We will proceed by induction on the length of $u$.

For $u=\emptyset$, we will assign $p_\emptyset$ as the empty position, $B_\emptyset = B(0, r)$,
and we let $\mu(N_\emptyset) = 1$.  For $u=(s \concat i, t \concat j)$ where $i< \ell$,
let $k_u=\abs{E_{\lh{s}}^i \cap B_{(s, t)}}$.  If $j > k_u$, or if $j=0$, then $u$ is \emph{inappropriate},
otherwise, we let $p_u$ be the position of length $\lh{p_{(s, t)}}+2$
in which player $\I$ has played $E_{\lh{s}}^i(t, j)$ where
\[E_{\lh{s}}^i(t, j) = \begin{cases}E_{\lh{s}}^i \cap B_{(s, t)} & j=k_u\\ 
E_{\lh{s}}^i \cap B_{(s, t)} \setminus \bigcup_{j<j'\leq k_u}\setof{\tau(E_{\lh{s}}^i(t, j'))} & j < k_u\end{cases}\]
and player $\II$ has followed $\tau$.  If $x_u$ is the point chosen by $\tau$, then we let 
\[B_u=B(x_u, (1-\beta_{\lh{u}})\rho_{\lh{u}})\]

The idea is that we first choose which $E_n^i$ to play from, depending on $i$, and then,
if $j$ is nonzero and less than or equal to $k_u$ (which is the maximum number of points we could legally play from $E_n^i$),
we play exactly $j$ many points of $E_n^i$.  Which points of $E_n^i$ we play is decided by $\tau$.
We remove $\tau$'s favorite points first.

We assign $\mu$-value to $N_{u}=N_{(s \concat i, t \concat j)}$ by the rule
\[\mu(N_{(s \concat i, t \concat j)}) = \frac{1}{\ell_u j^{(1+\epsilon)} \sum_{j' = 1}^{k_u}
\frac{1}{{(j')}^{(1+\epsilon)}}} \mu(N_{(s, t)})\]
where $\ell_u=\abs{\setof{i' < \ell \suchthat \exists j' ~ k_{(s \concat i', t \concat j')} >0}}$
is the number of possible choices for $i$ which yield appropriate choices for $u$.
Letting $\epsilon>0$ be arbitrary.

Note that for appropriate $u = (s \concat i, t \concat j)$, we have
\[\mu(N_{(s \concat i, t \concat j)}) = \frac{1}{\ell_u j^{(1+\epsilon)} \sum_{j' = 1}^{k_u}
\frac{1}{{(j')}^{(1+\epsilon)}}} \mu(N_{(s, t)}) \geq c j^{-(1+\epsilon)}\mu(N_{(s, t)})\]
where $c = \frac{1}{\ell \sum_{j'=1}^\infty (j')^{-(1+\epsilon)}}$.
And so we have, for appropriate $(s, t) \in (\ell \times \omega)^{<\omega}$
\[\mu(N_{(s, t)}) \geq c^{\lh{t}} \prod_{m=0}^{\lh{t}-1} t(m)^{-(1+\epsilon)}.\]

The reason we need to use $\epsilon$ here is because $k_u$ may go to infinity with $\abs{u}$,
and so we have no uniform constant $c$ without using a summable series.

Now let $x \in A \cap B(0, r)$.  Since the sets $E_n$ are maximal $\frac{1}{2}\rho_n$ separated,
there is a sequence of points $x_n \in E_n$ so that for every $n$, $\abs{x_n-x}<\frac{1}{2}\rho_n$.
Thus since $\rho_{n+1}=\beta_{n}\rho_n$ and $\beta_n<\frac{1}{2}$, we have 
\[\abs{x_{n+1} - x_n} \leq \abs{x_{n+1} - x} + \abs{x_{n}-x} < \frac{1}{2}\rho_{n+1} +
\frac{1}{2}\rho_n = \rho_n\left(\frac{1}{2} + \frac{1}{2}\beta_n\right)< \rho_n \left(1 - \beta_n\right)\]
So that each $x_{n+1}$ is a legal possibility following $x_n$.
Because of this, we can obtain sequences $i_n$ and $j_n$ so that for every $n$,
$x_n \in E_n^{i_n}(j_0, \dots j_n)$ and $x_n = \tau(E_n^{i_n}(j_0, \dots j_n))$.
Since $\tau$ is a winning strategy, and $x$ is in player $\I$'s target set, and each move
we made for player $\I$ was legal, it must be the case that player $\I$'s condition on the
number of choices offered is violated, i.e. for every constant $C$
\[\limsup_{n \to \infty} \frac{{\left(\prod_{m \leq n} j_m\right)}^{-1}}
{\left(\prod_{m\leq n} \beta_m\right)^\delta} > C.\]
Now we can compute, for any $\gamma> \delta(1+\epsilon)(1+\eta) >\delta(1+\epsilon)$

\begin{equation*}
\begin{split}
\limsup_{r \to 0}\frac{\mu(B(x, r))}{r^\gamma} &\geq \limsup_{n \to \infty}\frac{\mu(B(x, 2\rho_n))}{(2\rho_n)^\gamma}~\text{(because it is a $\limsup$)}\\
&\geq \limsup_{n \to \infty}\frac{\mu(B(x_n, \rho_n)}{(2\rho_n)^\gamma}~\text{(monotonicity)}\\
&\geq \limsup_{n \to \infty}\frac{\mu(N_{(i_0 \dots i_n, j_0 \dots j_n)})}{(2\rho_n)^\gamma}~\text{(push-forward and monotonicity)}\\
&\geq \limsup_{n \to \infty}\frac{c^{n+1}\prod_{m\leq n} j_m^{-(1+\epsilon)}}{(2\rho_0\prod_{m<n}\beta_m)^\gamma}~\text{(by the definition of $\mu$)}\\
&\geq \limsup_{n \to \infty}\frac{c^{n+1}}{(2\rho)^\gamma}{\left(\frac{\prod_{m\leq n} j_m^{-1}}{(\prod_{m\leq n} \beta_m)^\delta}\right)}^{(1+\epsilon)} \left(\frac{(\prod_{m\leq n} \beta_m)^{\delta(1+\epsilon)}}{(\prod_{m<n} \beta_m)^{\gamma}}\right)
\\ &
\end{split}
\end{equation*}

\begin{equation*}
\begin{split}
\phantom{\limsup_{r \to 0}\frac{\mu(B(x, r))}{r^\gamma}} &
\geq \frac{C^{(1+\epsilon)}}{(2\rho_0)^\gamma}\limsup_{n \to \infty}c^{n+1} \left(\frac{(\prod_{m\leq n} \beta_m)^{\delta(1+\epsilon)}}{(\prod_{m<n} \beta_m)^{\gamma}}\right)~\text{(player $\I$ lost)}\\
&= \frac{C^{(1+\epsilon)}}{(2\rho_0)^\gamma}\limsup_{n \to \infty}c^{n+1} \beta_n^{\delta(1+\epsilon)} \left(\prod_{m < n}\frac{\beta_m^{\delta(1+\epsilon)}}{\beta_m^{\gamma}}\right)\\
&= \frac{C^{(1+\epsilon)}}{(2\rho_0)^\gamma}\limsup_{n \to \infty}c^{n+1} \beta_n^{\delta(1+\epsilon)} \left(\prod_{m < n}\beta_{m}^{\delta(1+\epsilon)-\gamma}\right)\\
&\geq \frac{C^{(1+\epsilon)}}{(2\rho_0)^\gamma}\limsup_{n \to \infty}c^{n+1} (\prod_{m < n} \beta_m^{\eta})^{\delta(1+\epsilon)} \left(\prod_{m < n}\beta_{m}^{\delta(1+\epsilon)-\gamma}\right)~\text{(equation~(\ref{eqn:betaspeed}))}\\
&= \frac{C^{(1+\epsilon)}}{(2\rho_0)^\gamma}\limsup_{n \to \infty}c^{n+1} \prod_{m < n} \beta_m^{\delta(1+\epsilon)(1+\eta) - \gamma}
\end{split}
\end{equation*}
Now since $c$ is a fixed constant and $\beta_i \to 0$, and since $\delta(1+\epsilon)(1+\eta) < \gamma$,
we have that for large enough $m$, $\frac{c}{{\beta_m}^{\gamma - \delta(1+\epsilon)(1+\eta)}}> 2$ and so
\[\limsup_{n \to \infty}c^{n+1} \prod_{m < n} \beta_m^{\delta(1+\epsilon)(1+\eta) - \gamma} = \infty\]

Thus by Theorem~\ref{thm:rtt}, we have shown that for any $\gamma>\delta$, $\HD(A) \leq \gamma$, and so $\HD(A) \leq \delta$.
\end{proof}

\begin{corollary}[$\mathsf{AD}$]
Let $A \subseteq \mathbb{R}^d$ and $0 \leq \delta \leq d$, then either $A$ contains a
compact set $K$ so that $\HD(K)\geq\delta$ or $\HD(A) \leq \delta$.
\end{corollary}

To illustrate some of the issues concerning the determinacy of this game we
consider the following examples.

\begin{example} \label{ex1}
Let $0 < \delta \leq 1$  $K_n \subseteq (n, n+1)$ be a compact set with
$\HD(K_n) = \delta\left(1-\frac{1}{n+1}\right)$ for each $n \in \omega$.
Then player $\II$ wins $G^\delta_{\vec\beta}(\bigcup_n K_n)$, since this is a determined game,
and player $\I$ cannot win, since $\bigcup_n K_n$ doesn't contain any compact subsets of Hausdorff dimension $\delta$.
\end{example}

\begin{example}
Let $B \subseteq \mathbb{R}$ be a Bernstein set, then since $\lambda^*(B)>0$, we must have $\HD(B) = 1$.
Clearly player $\I$ cannot win $G^\delta_{\vec \beta}(B)$ for any $\delta>0$, since $B$ cannot contain
any uncountable closed set, so in particular $B$ cannot contain any compact set with positive Hausdorff dimension.
On the other hand, player $\II$ cannot have a winning strategy.

To see this, suppose player $\II$ had a winning strategy $\tau$ in $G^\delta_{\vec\beta}(B)$,
then one can construct a perfect subset of $\mathbb{R} \setminus B$ by building inductively
a perfect set of runs following $\tau$ where at 
at step $n$ we consider $x=\tau(E_0, \dots, E_{n-1},E)$ for some maximal legal move $E$,
and also $x'= \tau(E_0, \dots, E_{n-1},E \setminus \setof{ x})$. 
Since these moves are maximal, playing them doesn't violate player $\I$'s requirement
($\I$ is playing approximately $\frac{1}{\beta_{n-1}}$ many sets at round $n$, so is
satisfying the rule for any $\delta\leq 1$, that is, for all $\delta$).
This gives a perfectly splitting tree of positions, in which each level corresponds to disjoint closed intervals.
And since player $\I$'s condition is met, all branches through this tree must result
in points in $\mathbb{R} \setminus B$, which is impossible.
\end{example}

\section{The Unfolded Game}

In this section, we introduce an unfolded version of the Hausdorff
dimension game, and show that it is equivalent to the original.  This
result gives that analytic sets have the property that they can be
approximated from the inside by compact sets of the appropriate
Hausdorff dimension.  This is interesting, as other proofs of this
property are generally quite involved and require the analysis of the
approximations to the Hausdorff outer measure, and a so-called
``increasing sets lemma'' (again, see Theorems~47, 48 of \cite{RogersBook}),
and these are completely absent from our proof.

First a simple combinatorial lemma.

\begin{lemma}\label{lem:linord}
Suppose $A$ is a finite set with linear orders $\preceq_1, \preceq_2, \dots, \preceq_n$.
There is an element $a \in A$ so that for every $i \leq n$
\[\abs{\setof{b \in A \suchthat b \preceq_i a}}\geq \frac{1}{n}\abs{A}\]
\end{lemma}
\begin{proof}
Let 
\[A_i = \setof{a \in A \suchthat \abs{\setof{b \in A \suchthat b \preceq_i a}} < \frac{1}{n} \abs{A}}\]

Suppose the lemma fails, so that $\bigcup_i A_i = A$.
We will proceed by counting: First note that $A_i$ is an initial segment of $A$ by $\preceq_i$,
since if $a \in A_i$ and $b \preceq_i a$, then certainly
$\setof{c \in A \suchthat c \preceq_i b} \subseteq \setof{c \in A \suchthat c \preceq_i a}$.
So for each $i$, there is some $a_i \in A_i$ so that $A_i = \setof{a \in A \suchthat a \preceq_i a_i}$
But then since $a_i \in A_i$, we have that

\[\abs{A_i} = \abs{\setof{a \in A \suchthat a \preceq_i a_i}} < \frac{1}{n} \abs{A}\]
and so since $A = \bigcup_i A_i$ we have
\[\abs{A} \leq \sum_{i\leq n} A_i < \sum_{i\leq n} \frac{1}{n} \abs{A} = \abs{A}\]
a contradiction.
\end{proof}

\begin{definition} \label{def:hdg}
Let $d \geq 1$ be an integer and fix $\rho_0>0$, $0 < \beta_{i+1}\leq\beta_i < \frac{1}{2}$
be so that $\lim_{i \to \infty} \beta_i = 0$ satisfying 
\begin{equation}\label{eqn:betaspeed}
\forall \eta>0\ \exists n_0\ \forall n \geq n_0\ \beta_n \geq \prod_{i<n} \beta_i^\eta.
\end{equation}
Define $\rho_n=\left(\prod_{i<n} \beta_i\right) \rho_0$.
Let $F \subseteq \mathbb{R}^d \times \omega^\omega$ and let
$A=p[F]=\setof{x \in \mathbb{R}^d \suchthat \exists y \in \omega^\omega (x, y) \in F}$.
For $0 <\delta \leq d$, the \emph{unfolded $\delta$-Hausdorff dimension game} with target set $A$ is following game:

Player $\I$ makes moves $F_i$ in each round $i$, and whenever $\I$ chooses,
they may play also a digit $y_j \in \omega$ extending the finite sequence
$y_0, \dots y_{j-1}$ of any digits played so far.
As before, player $\II$ simply makes moves $x_i \in F_i$.

In order to not lose trivially, 
player $\I$ must ensure that the following hold.
\begin{itemize}
\item $F_i$ is a finite set of points in $\mathbb{Q}^d$.
\item $F_i$ is $3\rho_i$ separated. 
\item $F_{i+1} \subseteq B(x_i, (1-\beta_{i})\rho_i)$.
\item There exists some $c>0$ so that $\displaystyle \limsup_{n \to \infty}
\frac{\prod_{i < n} \lh{F_i}^{-1}}{{\prod_{i<n}\beta_i^\delta}} \leq c$.
\item For every $j \in \omega$, $y_j$ was eventually played.
\end{itemize}
Provided player $\I$ meets these requirements, player $\I$ wins if and only if 
\[\left(x, y\right) \in F\]
where $x = \lim_{n \to \infty} x_n$.
\end{definition}

Clearly if player $\I$ has a winning strategy in the unfolded
$\delta$-Hausdorff dimension game, then $\I$ has a winning strategy in
the original $\delta$-Hausdorff dimension game, since the unfolded
version has an extra requirement of producing a witness, and so is a
strictly harder game for $\I$.  Our goal then, is to show that we can
take a player $\II$ strategy in the unfolded game, and use it to
construct a strategy in the original.  Unfortunately, this requires us
to give up a little ground in the dimension.

\begin{theorem}\label{thm:unfthm}
If player $\II$ has a winning strategy in the unfolded $\delta$-Hausdorff dimension game,
and $s > \delta$, then player $\II$ has a winning strategy in the $s$-Hausdorff dimension game. 
\end{theorem}

We will prove a technical lemma that will be central for the argument.  For this purpose, we need a little notation.

\begin{notation}
Let $p=(F_0, x_0, F_1, x_1, \dots, F_n, x_n)$ be a position in the Hausdorff dimension game.
Let $q=(E_0, x_0', E_1, x_1', \dots, E_n, x_n')$ be a position of the unfolded Hausdorff dimension
game in which the digits of the finite sequence $u$ have been played along
with the $E_i$ sets (in some subsequence of the rounds).
We'll call $q$ a \emph{simulation of $p$ with partial witness $u$}
if for each $i$, $E_i \subseteq F_i$ and $x_i=x_i'$.
\end{notation}
And now we are ready to state our main technical lemma
\begin{lemma}\label{lem:extwit}
Let $\tau$ be a strategy in the unfolded Hausdorff dimension game and let $p$ be a position
of the Hausdorff dimension game.  Suppose we have some finite sequence of
partial witnesses $u_0, \dots, u_n$ and a finite sequence of positions $q_0, \dots, q_n$
of the unfolded Hausdorff dimension game so that for each $i$, $q_i$ is
a simulation of $p$ with partial witness $u_i$ so that $q_i$ is consistent with $\tau$.  

Given 
\begin{enumerate}
\item any finite sequence $v_0, \dots, v_n$ so that for each $i$,
either $v_i=u_i$ or $v_i$ is an extension of $u_i$ by a single extra digit,
\item and any move $F$ for $\I$ which is legal at $p$, 
\end{enumerate}
there is some $x \in F$ so that for each $i$, there is an extension $q_i'=q_i \conc E_i \conc x$
which is a simulation of $p \conc F \conc x$ with partial witness $v_i$
so that $\abs{E_i} \geq \frac{1}{n+1} \abs{F}$, and so that $q_i'$ is also consistent with $\tau$.
\end{lemma}
\begin{proof}
We will apply Lemma~\ref{lem:linord} more or less directly to obtain this result.
Given $v_i$ either extending $u_i$ or identical to $u_i$, we have that $\tau$ at position $q_i$
induces a linear order $\preceq_i$ on $F$ by $\tau$'s preference of which point to choose
in response to the move $F$ where the extra digit (if any) of $v_i$ is offered.
More precisely, define for each $x \in F$ the rank $r_i(x)$ by:
\[r_i(x) = \abs{F} \Leftrightarrow x = \tau(q_i \conc (F, v_i))\]
\[r_i(x) = j \Leftrightarrow x = \tau(q_i \conc (F \setminus \setof{y \suchthat r(y)>j}, v_i))\]
and the linear ordering $\preceq_i$ by
\[x \preceq_i y \Leftrightarrow r_i(x)\leq r_i(y)\]

Note that by the definition of $\preceq_i$, $\tau$ will always pick the
maximal element of any $\preceq_i$-initial segment offered to it, i.e.
$\tau(q_i \conc (F_{\preceq_i z}, v_i))=z$ where $F_{\preceq_i z} = \setof{y \in F \suchthat y \preceq_i z}$.
By Lemma~\ref{lem:linord}, there is some $x \in F$ so that for each $i$, $\tau$
will pick $x$ in response to the move $E_i = \setof{y \in F \suchthat y \preceq_i x}$,
and $\abs{E_i} \geq \frac{1}{n+1}\abs{F}$.
\end{proof}

Note that in Lemma~\ref{lem:extwit}, we did not require the $u_i$ to be distinct.
This will make our application of the lemma easier, when we choose to split a partial
witness $u$ into several extensions, and still keep $u$ itself alive.

\begin{proof}[Proof of Theorem~\ref{thm:unfthm}]
Suppose $\tau$ is a winning strategy in the unfolded
$\delta$-Hausdorff dimension game, and let $s>\delta$.  We first
attempt to motivate the proof: We want to construct a strategy for
which every full run has a tree of simulations consistent with $\tau$
for all possible witnesses.  The main obstacle is to make sure that
along each branch of this tree, we've offered $\tau$ enough choices so
that the branch is not winning for trivial reasons.  Then we can use
that $\tau$ is winning to prove that $(x, y) \not \in F$ for every
$y$, thus $x \not \in A$, producing a win in the original Hausdorff
dimension game.  In order to maintain that all the simulations are
consistent with $\tau$, we need to play fewer sets when copying $\I$'s
moves, and so we need to be able to absorb the extra $\frac{1}{n}$
factor in each round that we have $n$ partial witnesses.  This is
where the fact that $\beta_i \to 0$ is critical.

Enumerate $\omega^{<\omega}$ as $\setof{w_i \suchthat i \in \omega}$ so that
if $w_j \subseteq w_i$, then $j \leq i$.
For a while, play according to $\tau$ in the $s$-Hausdorff dimension game,
playing no witness moves at all, until $\beta_i$ gets small enough so that 
\[\beta_i^{s-\delta} < \frac{1}{2}\]
at which point we can absorb a factor of $\frac{1}{2}$.
Now we apply Lemma~\ref{lem:extwit} to the current position with witnesses $u_0=u_1=\emptyset=w_0$
and $v_0=u_0$, $v_1=w_1$.  Continue play in every round afterwards applying
Lemma~\ref{lem:extwit} with $u_0=w_0=v_0$, $u_1=w_1=v_1$.
Note that this maintains the hypotheses of Lemma~\ref{lem:extwit},
so that we can continue to apply it.  We do this until $\beta_i$ is small enough so that
\[\beta_i^{s-\delta} < \frac{1}{3}\]
at which point we can absorb a factor of $\frac{1}{3}$.  We would like to add the witness $w_2$
to our list at this point, and we know that $w_2$ must extend either $w_0$ or $w_1$,
and so we apply Lemma~\ref{lem:extwit} to three witnesses $u_0=w_0$, $u_1=w_1$, $u_2=w_2 \res \lh{w_2}-1$,
in which the ancestor of $w_2$ appears twice, with $v_0=u_0$, $v_1=u_1$ and $v_2=w_2$.
It is clear that we can continue this algorithm to define a strategy in the $s$-Hausdorff dimension game.
We now demonstrate that it does the job:

Suppose $x$ is the result of our strategy, and suppose for the sake of a contradiction that $x$ was a loss for us.
In other words $x \in A$ and there exists some $c>0$ so that 
\[\limsup_{n \to \infty} \frac{\prod_{i < n} \lh{F_i}^{-1}}{{\prod_{i<n}\beta_i^s}} \leq c\]
where $F_i$ are the sets $\I$ played along the way.
By the definition of $A$, we have that for some $y \in \omega^\omega$, $(x, y) \in F$.
Let $i_0, i_1, \dots$ be the subsequence so that $w_{i_n} = y \res n$.
We have a simulation by $\tau$ in which $\I$ eventually plays all the digits of $y$, say with sets $E_i$.
Our job now is to show that
\[\limsup_{n \to \infty} \frac{\prod_{i < n} \lh{E_i}^{-1}}{{\prod_{i<n}\beta_i^\delta}} <\infty\]
Note that by the construction of our strategy, we were always able to play $E_i \subseteq F_i$ so that
\[\abs{E_i} \geq \beta_i^{s-\delta} \abs{F_i}\]
so then we have
\begin{align*}
\limsup_{n \to \infty} \frac{\prod_{i < n} \lh{E_i}^{-1}}{{\prod_{i<n}\beta_i^\delta}} & \leq  \limsup_{n \to \infty}
\frac{\prod_{i < n} (\beta_i^{s-\delta} \abs{F_i}) ^{-1}}{{\prod_{i<n}\beta_i^\delta}}
\\ & = \limsup_{n \to \infty} \frac{\prod_{i < n} \beta_i^{\delta-s} \abs{F_i} ^{-1}}{{\prod_{i<n}\beta_i^\delta}}
\\ & = \limsup_{n \to \infty} \frac{\prod_{i < n} \abs{F_i} ^{-1}}{{\prod_{i<n}\beta_i^s}} 
\\ & < \infty
\end{align*}

\end{proof}

\section{Applications of Unfolding} \label{sec:aou}

In this section we derive some consequences of the Hausdorff dimension game as well as the
unfolding theorem, Theorem~\ref{thm:unfthm}. Our first application concerns the existence of continuous
uniformizations. Recall first the situation with regards to measure and category. Assuming $\ad$, if $R \subseteq 
X\times \ww$  and $\dom(R)$ is comeager, then there is a comeager set $C\subseteq \dom(R)$
and a continuous function $f \colon C\to \ww$ which uniformizes $R$, that is, for all $x \in C$
we have $R(x,f(x))$. This continuous uniformization phenomenon is an important aspect of category
arguments and follows from an unfolding argument for category (using the Banach-Mazur game, also
known as the $**$-game). There is also a corresponding theorem for measure. Again assuming $\ad$, if $R\subseteq X\times \ww$
and $\dom(R)$ has measure $1$ with respect to some Borel probability measure $\mu$, then for any
$\epsilon >0$ there is a $A\subseteq X$ with $\mu(A)>1-\epsilon$ and a continuous $f\colon A\to \ww$
which uniformizes $R$. Our unfolding result Theorem~\ref{thm:unfthm} allows us to get a similar result for Hausdorff dimension.

\begin{theorem} \label{thm:cut}
Assume $\ad$. Suppose $R\subseteq \R^d \times \ww$ and $\dom(R)$ has Hausdorff dimension at least $\delta$.
Then for any $\delta'<\delta$ there is a $B\subseteq \dom(R)$ with $\HD(B)\geq \delta'$ and a
continuous $f \colon B\to \ww$ which uniformizes $R$.
\end{theorem}

\begin{proof}
Fix $\delta'<\delta''<\delta$ and consider the unfolded $\delta'$-Hausdorff dimension game as in Definition~\ref{def:hdg}
for the set $A=\dom(R)$, and using $R$ for the set $F$. Here we use a fixed sequence $\{ \beta_i\}$
satisfying the conditions of Definition~\ref{def:hdg}. By $\ad$ this game is determined. If $\II$
had a winning strategy for this unfolded game, then by Theorem~\ref{thm:unfthm} $\II$ would have a
winning strategy for the (regular non-unfolded) $\delta''$-Hausdorff dimension game.
From Theorem~\ref{thm:p2} we have that $\HD(A)\leq \delta''$, a contradiction.
Thus, $\I$ has a winning strategy $\sigma$ for the unfolded $\delta'$-Hausdorff dimension game. 
Ignoring the witness moves that $\sigma$ makes, $\sigma$ gives a strategy $\bar{\sigma}$ for the
regular $\delta'$-Hausdorff dimension game for $A$. The proof of Theorem~\ref{thm:p1}
gives a compact set $K\subseteq A$ with $\HD(K)\geq \delta'$. For $x \in K$ there is a unique
run according to $\sigma$ which produces the point $x$ (that is, $\lim x_n=x$). Let $y=(y_0,y_1,\dots)$
be the sequences of witness moves played by $\sigma$ along this run. Then $R(x,y)$ as $\sigma$
is winning for $\I$. If for $x \in K$ we let $f(x)$ be this $y$, then the function $f$ is continuous on $K$
as $y \res k$ is determined by some finite part $p$ of this run by $\sigma$, and any $x'$ in $K$ which is in
the same open set determined by $x_n$ (where $n$ is the length of $p$) will have $f(x')\res k=f(x)\res k$.

\end{proof}

We note that although Theorem~\ref{thm:cut} is stated under $\ad$ as a hypothesis, the determinacy assumption
is entirely local, we just need the determinacy of the unfolded game.  So, for example, if
$R\subseteq \R^d \times \ww$ is $\bS^1_1$, then we just need the determinacy of $\bD^0_3$ games, which is theorem
of $\zf$ (the condition that each digit $y(i)$ is eventually played is a $\bP^0_2$
condition, and the $\limsup$ condition on the size of $\I$'s moves is a $\bS^0_2$ condition).
In particular, projective determinacy $\pd$ is enough to get the conclusion of
Theorem~\ref{thm:cut} for all projective relations $R$.

Within the realm of $\ad$, another important result about measure and category is the full additivity
of these notions. That is, any wellordered union (of any length) of meager sets is meager, and
likewise for measure zero sets. These results can be proved ether using the Fubini theorem (or
Kuratowski-Ulam theorem in the case of category) or by an argument using an unfolded game.
In the case of Hausdorff measure, we do not have an analog of the Fubini theorem. However,
our unfolding theorem can be used to prove the corresponding result.

\begin{theorem} \label{thm:wou}
Assume $\ad$. Then any wellordered union of subsets of $\R^d$, each of
which has Hausdorff dimension at most $\delta$, has Hausdorff dimension at most $\delta$.
\end{theorem}

\begin{proof}
Let $A=\bigcup_{\alpha<\theta} A_\alpha$ where $A_\alpha \subseteq \R^d$ and $\HD(A_\alpha)\leq \delta$.
Suppose $\HD(A)>\delta$. 
We may assume $\theta$ is least so that $\HD(\bigcup_{\alpha<\theta} A_\alpha)$ has Hausdorff dimension greater than
$\delta$, and thus we may assume that the sequence $A_\alpha$ in increasing. Fix any $\delta'$ with $\delta <\delta' <\HD(A)$. 
From Theorem~\ref{thm:p2} it suffices to show that $\II$ wins the
$\delta'$-Hausdorff dimension game for $A$. Suppose not, and let $\sigma$ be a winning strategy for $\I$
in the $\delta'$-Hausdorff dimension game for $A$.

Suppose first that $\cof(\theta)=\omega$, and let $\alpha_n<\theta$ be such that $\sup_n \alpha_n=\theta$.
So, $A=\bigcup_n A_{\alpha_n}$. Since each $A_{\alpha_n}$ has Hausdorff dimension
$\leq \delta$, $\mathcal{H}^{\delta'}(A_{\alpha_n})=0$ for each $n$. As $\mathcal{H}^{\delta'}$ is a measure,
$\mathcal{H}^{\delta'}(A)=0$, a contradiction.

Suppose next that $\cof(\theta)>\omega$. Note that $\theta<\Theta$ as $\Theta$ is the supremum of the lengths
of the prewellorderings of $\R$ (or equivalently, the lengths of the increasing sequences of subsets
of $\R$). A theorem of Steel (see Theorem 1.1 of \cite{Jackson2001} for the general statement and proof)
says that, assuming $\ad$,  for any $\theta<\Theta$ with $\cof(\theta)>\omega$ there is a 
$\varphi\colon B\to \theta$ (for some set $B\subseteq \ww$)
which is onto and such that any $\bS^1_1$ set $S\subseteq B$ is bounded in the prewellordering,
that is, $\sup \{ \varphi(x) \colon x \in S\} <\theta$.
Fix such a map $\varphi$ for the ordinal $\theta$. 
Consider the relation $\R\subseteq \R^d \times \ww$ given by
\[
R(x,y) \leftrightarrow (x \in A) \wedge (y \in B) \wedge (x \in A_{\varphi(y)})
\]
Clearly $\dom(R)=A$. Consider the unfolded $\delta'$-Hausdorff dimension games for $R$. From Theorem~\ref{thm:unfthm}
we have that $\II$ cannot win this game as otherwise we would have that $\HD(A) \leq \delta'$
(since for every $\delta'' >\delta'$, from theorem~\ref{thm:unfthm}, $\II$ would win the
regular $\delta''$ game for $A=\dom(R)$, and so $\HD(A)\leq \delta''$). 
As we are assuming $\ad$, we may fix a winning strategy $\sigma$ for $\I$ in this unfolded game.
Let $\sigma_1$ be the strategy which extracts the $y=(y_0,y_1,\dots)$ moves from the play by $\sigma$.
Let $S=\sigma_1 [\ww]$, more precisely, let $S$ collect all of the $y$'s which come from a any run of $\sigma_1$
in which $\II$ has followed the rules of the game. Clearly $S$ is $\bS^1_1$, and so there is an
$\alpha <\theta$ such that $\sup \{\varphi(y)\colon y \in S\}<\alpha$. This says that
$\I$ wins the unfolded $\delta'$-Hausdorff dimension game for the set $A_\alpha$. Thus, $\I$
wins the regular $\delta'$-Hausdorff dimension game for $A_\alpha$ and so $\HD(A_\alpha)\geq \delta'$, a
contradiction.

\end{proof}

The arguments of the current paper naturally suggest two questions. First, can we
get a characterization for when player $\I$ or $\II$ wins $G^\delta_{\vec \beta}(A)$
when $\HD(A)=\delta$? As we noted in Example~\ref{ex1}, either player could win in this case
(just given $\HD(A)=\delta$). Second, to which class of metric spaces can we extend our
basic results (Theorems~\ref{thm:p1},\ref{thm:p2}, and \ref{thm:unfthm})?

\bibliographystyle{amsplain}

\bibliography{HDG}

\end{document}